\documentclass{article}

\usepackage{amssymb, amsmath, amsthm, graphicx,mathrsfs} 
\usepackage{xcolor}
\usepackage{fullpage}
\usepackage[colorlinks=true,linkcolor=teal,citecolor=teal,filecolor=teal,urlcolor=teal]{hyperref}
\usepackage{cleveref}[hyperref]
\usepackage{multirow}
\usepackage[backend=biber,style=numeric,sorting=nty,  maxnames=99,
  minnames=99, giveninits=true]{biblatex}
\renewbibmacro{in:}{}
\bibliography{citations}

\allowdisplaybreaks

\newtheorem{thm}{Theorem}[section]

\newtheorem{conj}[thm]{Conjecture}
\newtheorem{quest}[thm]{Question}
\newtheorem{obs}[thm]{Observation}

\newtheorem{claim}[thm]{Claim}

\def\cov{{\rm{cov}}}

\date{\today}

\newcommand{\ceil}[1]{\left\lceil #1 \right\rceil}

\title{Covering complete $r$-partite hypergraphs with few monochromatic components}
\date{\today}
\author{
{{Luke Hawranick}}\thanks{
\footnotesize {University of South Carolina, Columbia, SC 29208, USA. Email: {\tt hawranil@email.sc.edu}.}}
\and{{Ruth Luo}}\thanks{University of South Carolina, Columbia, SC 29208, USA. Email: {\tt ruthluo@sc.edu}. Research of this author
is supported in part by NSF grant DMS-2452134.
}}
\begin{document}
\maketitle

\vspace{-0.3in}

\begin{abstract}An edge-coloring of a hypergraph is {\em spanning} if every vertex sees every color used in the coloring.
In this paper, we prove that for $k \geq 2r \geq 6$, in any spanning $k$-coloring of the edges of a complete $r$-partite $r$-uniform hypergraph $H$, the vertices of $H$ can be covered by a set of at most $k-r+1$ monochromatic connected components. This proves a conjecture of Gy\'arf\'as and Kir\'aly which is related to a special case of Ryser's conjecture. We also prove that for $k \in \{2,3\}$, every spanning $k$-edge-coloring of a complete bipartite graph admits a covering of its vertices using at most $k$ monochromatic components.
\end{abstract}

\section{Introduction}

A \emph{vertex cover} of a hypergraph $H$ is a subset $S$ of vertices such that for every edge $e$ in $H$, we have $e \cap S \ne \varnothing$. A \emph{matching} of $H$ is a set of pairwise disjoint edges. For every $r$-uniform hypergraph $H$, the size of a minimum vertex cover is at most $r$ times as large as the size of a maximum matching. Indeed, every edge of $H$ intersects an edge of a maximum matching. This bound is sharp, as seen by the hypergraph of all $r$-sets from a ground set of size $rt-1$ where $t \in \mathbb N$. A well-known conjecture of Ryser states that if $H$ is $r$-partite, then the multiplicative factor in the bound can be improved to $r-1$.
\begin{conj}[Henderson \cite{Henderson}, attributed to Ryser]\label{conj:ryser}
    Let $H$ be an $r$-partite, $r$-uniform hypergraph. The size of a minimum vertex cover of $H$ is at most $r-1$ times the size of a maximum matching of $H$.
\end{conj}
When $r=2$, the conjecture is equivalent to K\H{o}nig's Theorem for graphs. The conjecture is also proven for $r=3$ by Aharoni~\cite{Aharoni} and has partial progress for $r=4,5$ due to Haxell and Scott~\cite{HS}, but remains open in general. If true, Conjecture~\ref{conj:ryser} would be sharp for some values of $r$. For instance when a projective plane of order $r-1$ exists, we consider its corresponding hypergraph (in which points are vertices and lines are edges). The subhypergraph induced by removing a single vertex from the projective plane would be an extremal example \cite{TuzaARS}. Moreover, Abu-Khazneh, Bar\'at, Prokrovskiy, and Szab\'o ~\cite{ABPS} constructed a hypergraph which acheives the same bound when there exists a projective plane of order $r-2$. 

Motivated by the appearance of projective planes in lower bounds, there has been much work to solve the conjecture for hypergraphs in which the size of a maximum matching is $1$. Such hypergraphs are called \emph{intersecting}, as every pair of edges shares a vertex. This special case for intersecting hypergraphs is equivalent to the following conjecture about monochromatic components in an edge-colored complete graph.  \begin{conj}[Gy\'arf\'as \cite{Gyarfas}]\label{conj:equiv}
    In any coloring of the edges of the complete graph with $r$ colors, the number of monochromatic components required to cover $V(G)$ is at most $r-1$.
\end{conj}

Conjecture~\ref{conj:ryser} and Conjecture~\ref{conj:equiv} for intersecting hypergraphs are equivalent by the following reasoning \cite{Kiraly}. Suppose $H$ is an intersecting hypergraph from Conjecture~\ref{conj:ryser} with partition $V(H)=\bigcup_{i=1}^rV_i$. We construct a graph $G_H$ with $V(G_H)=E(H)$. For $u,v \in V(G_H)$, color $uv \in E(G_H)$ according to the smallest index of the parts $V_i$ in which the edges $u$ and $v$ meet in $H$. For every vertex $u$ of $G_H$ contained in a monochromatic component of $G_H$ in color $i$, the corresponding edge in $H$ contains a unique vertex of $H$ in part $V_i$, call it $u(i)$. If $uw$ is an edge in color $i$, then $u(i) = w(i)$. Moreover, for any $x \in V(G_H)$ contained in the same component as $u$ of color $i$, we have $x(i) = u(i)$ since this component is connected. Thus a set of $k$ monochromatic components that cover $V(G_H)$ corresponds to a vertex cover of $H$ of size $k$.

Now let $G$ be an $r$-colored complete graph from Conjecture~\ref{conj:equiv}. Arbitrarily order the monochromatic components of every color---let $C_{i,j}$ be component $j$ of color $i$. Construct a hypergraph $H_G$ with $V(H_G) = \{C_{i,j}\}_{i\in [r],j}$, and partition $V(H_G)$ according to the color of component $C_{i,j}$. For every $v \in V(G)$ and for every  color $i$, denote by $j_{v,i}$ the number of the component of color $i$ that contains $v$. Then add the $r$-tuple $\{j_{v,i}\}_{i \in [r]}$ to $E(H_G)$ for each vertex $v \in V(G)$. The $r$-uniform, $r$-partite hypergraph $H_G$ is intersecting, as for every $e,f \in E(H_G)$, there correspond vertices $u,v \in V(G)$. The edge $uv$ is colored with some color in $G$, say color $k$, implying that $u$ and $v$ lie in the same monochromatic component of color $k$. Hence, $e$ and $f$ intersect in $V_k$. Then a vertex cover $H_G$ of size $k$ corresponds to a set of $k$ monochromatic components that together cover all vertices in $G$.

We note that this conjecture has been proven for $r \le 5$ \cite{Tuza} by Tuza, but the cases $r \geq 6$ remain open. See~\cite{rysersurvey} for a survey of Ryser's conjecture and many of its variants.

\bigskip

Gy\'arf\'as \cite{Gyarfas} and Lehel \cite{Lehel} asked about a bipartite analogue of Conjecture~\ref{conj:equiv}. We use {\em biclique} to refer to any complete bipartite graph. It was conjectured that every $r$-edge-coloring of a biclique $G$ requires at most $2r-2$ monochromatic components to cover $V(G)$. This bound would be sharp---given by an example appearing in \cite{Gyarfas} and repeated in \cite{CFGLT}.

\begin{conj}[\cite{Gyarfas, Lehel}]\label{bicliqueconj}In every $r$-coloring of the edges of a biclique, the vertex set can be covered
by the vertices of at most $2r-2$ monochromatic components.
\end{conj}

Decades later, Chen, Fujita, Gy\'arf\'as, Lehel, and T\'oth~\cite{CFGLT} studied the problem more extensively, proving the conjecture for $r\leq 5$ and reducing the general case to design-type conjectures. 

We will discuss these problems in the more general setting of hypergraphs.
A {\em (connected) component} of a hypergraph $H$ is a maximal subhypergraph $H' \subseteq H$ such that every pair of vertices $u,v \in V(H')$ is connected by a sequence of edges in $H'$ in which the first edge contains $u$, the last edge contains $v$, and adjacent edges in the sequence intersect. Note that this sequence can be empty. Moreover, if the edges of $H$ are colored, then a {\em monochromatic component} of some color $c$ is a component of the subhypergraph of $H$ on the edges with color $c$.

Kir\'aly \cite{Kiraly} proved the analogue of Conjecture~\ref{conj:equiv} for $r$-uniform hypergraphs with $r \ge 3$. 

\begin{thm}[Kir\'aly \cite{Kiraly}]If $H$ is a $k$-edge-colored complete $r$-uniform hypergraph with $r \geq 3$, then at most $\ceil{k/r}$ monochromatic components are needed to cover $V(H)$, and this bound is sharp.
\end{thm}

\subsection{Spanning colorings and new results}

We note that Conjecture~\ref{conj:equiv} is obviously true for colorings in which the incident edges of some vertex receive only $r-1$ or fewer colors, as we can take the at most $r-1$ monochromatic stars centered at this vertex. Colorings in which, for every vertex, the incident edges receive all $r$ colors are called {\bf spanning}. In other words, in a spanning coloring, every vertex sees every color.

A {\em complete $r$-partite $r$-uniform} hypergraph $H$ is a hypergraph with vertex partition $V(H) = \bigcup_{i=1}^r V_i$ and $E(H) = \{\{v_1, \ldots, v_r\}: v_i \in V_i, i \in [r]\}$. 
Gy\'arf\'as and Kir\'aly studied the analogue of Conjecture~\ref{conj:equiv} for complete $r$-partite $r$-uniform hypergraphs with spanning colorings. 
 
Define the {\bf covering number}, $\cov(r,k)$, to be the minimum integer $\ell$ such that in any spanning $k$-coloring of the edges of any complete $r$-partite $r$-uniform $H$, $V(H)$ can be covered by at most $\ell$ monochromatic components. They proved the following result for the case $k \geq r+1$. 

\begin{thm}[Gy\'arf\'as, Kir\'aly~\cite{GK}]\label{GKthm} For all $t\geq 1$ and $r \geq 3$, $\cov(r,r+t) \geq t+1$. Moreover when $t \leq r-1$, $\cov(r,r+t) = t+1$. 
\end{thm}

For $r \ge 3$ and $k \le r$, they remark that $\cov(r,k) = 1$. The following was conjectured for the remaining cases.

\begin{conj}[Gy\'arf\'as, Kir\'aly~\cite{GK}]\label{GKconj} For all $t\geq 1$ and $r \geq 3$, $\cov(r,r+t) = t+1$.
\end{conj}

In this paper, we prove Conjecture~\ref{GKconj} by completing the case $t\geq r$.  
\begin{thm}\label{main}For all $t \geq r \geq 3$, $\cov(r,r+t) = t+1$.
\end{thm}

The lower bound $\cov(r,r+t) \geq t+1$ arises from a nontrivial construction in~\cite{GK}. We prove the upper bound. 

\medskip 

Conjecture~\ref{GKconj} is not true for $r=2$, i.e., bicliques. For instance, let $k = r+t = t+2$ and consider the biclique $X \cup Y$ with $X = \{x_1, \ldots, x_k\}$ and $Y = \{y_1, \ldots, y_k\}$ such that edge $x_iy_j$ receives color $j-i \pmod k$. In this spanning $k$-coloring, every color class is a perfect matching, hence every monochromatic component has exactly one edge. Then $k > t+1$ monochromatic components are needed to cover all vertices.

We note that the only known construction for the lower bound $2k-2$ in Conjecture~\ref{bicliqueconj} is a {\em non-spanning} coloring. The perfect matching construction mentioned above gives a spanning coloring implying $\cov(2,k) \geq k$. In~\cite{GK}, it is stated that Kir\'aly conjectures $\cov(2,k)$ should be at least $2k-4\sqrt{k}$, although we are not aware of any explicit constructions matching this bound. We prove that $\cov(2,k) = k$ for $k \in \{2,3\}$ using results from~\cite{CFGLT}. The remaining cases are open and likely difficult.
\begin{thm}\label{cov2k}
    For $k \in \{2,3\}$, $\cov(2,k) = k$.
\end{thm}

\begin{quest}
Determine $\cov(2,k)$ for $k \geq 4$.
\end{quest}

\medskip
{\bf Notation.} In our proofs, we consider a complete $r$-partite $r$-uniform hypergraph $H$ with parts $V_1, V_2, \ldots, V_r$ and an arbitrary spanning $(r+t)$-coloring with colors $c_1, c_2, \ldots, c_{r+t}$. The subhypergraph of edges of color $c_i$ partitions the vertices into some number of connected components. Choose an arbitrary ordering of the components, and let $H_{c_i, j}$ denote the $j$th connected component in color $c_i$. 

We associate to each vertex $v \in V(H)$ a vector $\vec{v} = (a_1, a_2, \ldots, a_{r+t}) \in [n]^{r+t}$ such that $v$ is contained in the $a_i$th component of color $c_i$, $H_{c_i, a_i}$, for all $1 \leq i \leq r+t$. We denote by $\vec v(i)$ the $i$th coordinate of $\vec v$. For a pair of  vertices $u$ and $v$, we denote by $\delta(u,v)$ the Hamming distance of their vectors $\vec{u}$ and $\vec{v}$, i.e., the number of colors in which $u$ and $v$ belong to different components.

\section{Proof of Theorem~\ref{main}}
Suppose that for some $r$-uniform, complete $r$-partite hypergraph $H$, there exists a spanning $(r+t)$-coloring in which $V(H)$ cannot be covered by $t+1$ monochromatic components. 

We will use the following claims, the first four of which were proven or used implicitly in~\cite{GK}. For completeness, we include their short proofs.

\begin{claim}\label{rsame} Let $v_1, \ldots, v_r$ be vertices with $v_i \in V_i$ for all $i \in [r]$. Then there exists a color $c_j$ in which vertices $v_1, \ldots, v_r$ are all contained in the same component.  That is, $\vec v_i(j) = \vec v_k(j)$ for all $i,k \in [r]$. 

\end{claim}
\begin{proof}
    Since $H$ is complete $r$-partite, $\{v_1, \ldots, v_r\} \in E(H)$ and it receives some color $c_j$.
\end{proof}

\begin{claim}\label{t+1diff}For any set of $t+1$ indices $j_1, j_2, \ldots, j_{t+1} \in [r+t]$ and any vector $(a_1, \ldots, a_{t+1}) \in [n]^{t+1}$, there exists a vertex $v \in V(H)$ such that $\vec v(j_i) \neq a_i$ for all $i \in [t+1]$. 
\end{claim}
\begin{proof}
    If there exists a set of $t+1$ indices and a vector $(a_1, \ldots, a_{t+1})$ such that for every vertex $v\in V(H)$, we have $\vec{v}(j_i) = a_i$ for some $i \in [t+1]$, then $V(H)$ is covered by at most $t+1$ monochromatic components: $H_{c_{j_1}, a_1},\ldots, H_{c_{j_{t+1}}, a_{t+1}}$, a contradiction.
\end{proof}

\begin{claim}\label{samepartdiffcolor}
    For any color $c_j$, and any part $V_i$, there exists at least $t+2$ vertices $v_1, \ldots, v_{t+2} \in V_i$ each in distinct components of color $c_j$. That is, $\vec v_i(j) \neq \vec v_k(j)$ for all distinct $i, k \in [t+2]$. 
\end{claim}
\begin{proof}
    Suppose for contradiction that $V_i$ is covered by at most $t+1$ components of color $c_j$. Let $\mathcal{C}$ be the set of these components. Fix $v \in V(H)-V_i$. Because the $(r+t)$-coloring of $E(H)$ is spanning and $H$ is $r$-partite, there is an edge in color $c_j$ containing both $v$ and a vertex $v_i$ of $V_i$. Thus, $\vec{v}_i(j) = \vec{v}(j)$, and therefore $v \in \mathcal{C}$. Hence, $V(H)$ is covered by $\mathcal{C}$, a contradiction.
\end{proof}

\begin{claim}\label{smalldist}
    If $u, v \in V(H)$ belong to different parts, then $\delta(u,v) \leq t+1$. 
\end{claim}

\begin{proof}
    Assume for contradiction that there exists $u,v \in V(H)$ in different parts such that $\delta(u,v) \ge t+2$. By relabeling parts, colors, and components, we may assume that $u \in V_r$, $v \in V_{r-1}$, $\vec{u} = (1,1,\ldots,1)$, and $\vec{v} = (1,1,\ldots,1,2,\ldots,2)$, where at most $r-2$ entries of $\vec{v}$ are $1$. By Claim~\ref{samepartdiffcolor}, for each $i \in [r-2]$, there exists $v_i \in V_i$ such that $\vec{v_i}(i) \ne 1$. Note that $\{v_i:i \in [r-2]\}\cup\{u,v\}$ is a set of $r$ vertices from pairwise distinct parts. We have $\vec{v_i}(i) \ne \vec{u}(i)$ for each $i \in [r-2]$ and $\vec{v}(j) \ne \vec{u}(j)$ for all $j > r-2$. This contradicts Claim~\ref{rsame}.
\end{proof}

\begin{claim}\label{distr}There exists vertices $u, v \in V(H)$ belonging to different parts such that $\delta(u,v) \geq r$.
\end{claim}
\begin{proof}
    Again, by relabeling colors and components, we may assume that there exists a vertex $a \in V_1$ `with $\vec a = (1,1, \ldots, 1)$. By Claim~\ref{t+1diff}, for any set of indices $J \subseteq [r+t]$ with $|J| = t+1$, we can find a vertex $b_J$ for which $\vec b_J(j_i) \neq 1$ for all $j_i \in J$. (It is possible that some of the $b_J$'s are the same vertices.) Then $\delta(a, b_J) \geq t+1 > r$. If some $b_J \notin V_1$, then $a$ and $b_J$ satisfy the claim and we are done. So suppose all such $b_J$ belong to $V_1$.

    Similarly, if any $b \in V_2$ satisfies $\delta(a,b) \geq r$, then we are done. So we may assume that each $b \in V_2$ contains at most $r-1$ coordinates not equal to $1$ in $\vec b$ and hence there exists a set $J$ of $r+t - (r-1) = t+1$ indices in which $\vec b$ is equal to $1$. Fix such a $\vec{b}$ and $J$. We obtain $\delta(b, b_J) \geq t+1$ by construction, and $b_J \in V_1$. Thus $b$ and $b_J$ satisfy the claim. 
\end{proof}

\subsection{The case $r=3$}

In this section we consider the case $t\geq r=3$. We color $H$ with $t+3$ colors.

For any three vertices $a,b,c$, we say a color $c_j$ is {\em $abc$-distinguishing} if $a$, $b$, and $c$ belong to distinct components in color $c_j$. Equivalently, $\vec a(j), \vec b(j)$, and $\vec c(j)$ are distinct.

\begin{claim}\label{distinguishing}
    For any $a \in V_1, b \in V_2, c \in V_3$, there is at most one $abc$-distinguishing color.
\end{claim}

\begin{proof}
    Suppose that for some $a\in V_1, b \in V_2, c \in V_3$, there are at least $2$ $abc$-distinguishing colors. Define \[I_{ab} = \{j: \vec a(j) = \vec b(j)\}.\] Similarly define $I_{bc}$ and  $I_{ac}$.
    For each $j \in I_{ab}$,  $a$ and $b$ belong to the same component of color $c_j$, $H_{c_j,\vec a(j)}$. Moreover, if $j$ belongs to two sets of $I_{ab}, I_{bc}, I_{ac}$, then $a,b,c$ all belong to the component $H_{c_j,\vec a(j)}$. Define
    \[\mathcal C_{ab} = \{H_{c_j, \vec a(j)}: j \in I_{ab}\}\]to be the collection of monochromatic components containing both $a$ and $b$. We similarly define $\mathcal C_{bc}$ and $\mathcal C_{ac}$. Then
    \[|\mathcal C_{ab} \cup \mathcal C_{bc} \cup \mathcal C_{ac}| = |I_{ab} \cup I_{bc} \cup I_{ac}| = t+3 - |\{j: c_j \text{ is } abc\text{-distinguishing}\}| \leq t+1.\]

    We will show that we can cover $V(H)$ with these at most $t+1$ monochromatic components. Consider any $v \in V_1$. The vectors $\vec v, \vec b, \vec c$ must all agree in some coordinate $j$. By definition, $j \in I_{bc}$, and thus $V_1 \subseteq \mathcal C_{bc}$. Similarly, $V_2 \subseteq \mathcal C_{ac}$ and $V_3 \subseteq \mathcal C_{ab}$. Since $V(H) = V_1 \cup V_2 \cup V_3$, the proof is complete.
\end{proof}
We are now ready to prove Theorem~\ref{main} for $r = 3$. Throughout the proof, one may refer to Figure~\ref{fig:r=3}.

\begin{proof}[Proof of Theorem~\ref{main} for $r=3$]
    By Claim~\ref{distr}, we may relabel parts of $V(H)$ such that there exists $v_1 \in V_1$ and $v_2\in V_2$ in which $\delta(v_1,v_2) \ge 3$. We may assume that $\vec{v_1}(1) = \vec{v_1}(2) = 1$, while $\vec{v_2}(1) = \vec{v_2}(2) = 2$. Because $t+1 \ge r+1 = 4$, it suffices to show that every vertex of $V(H)$ is covered by the union of the four components $H_{c_1,1}$, $H_{c_1,2}$, $H_{c_2,1}$, and $H_{c_2,2}$. Assume for contradiction that there is some $u \in V(H)$ in which $\vec{u}(1) = \alpha_1 \not \in \{1,2\}$ and $\vec{u}(2) = \alpha_2 \not \in \{1,2\}$. By Claim~\ref{distinguishing}, $u \not \in V_3$, otherwise, colors $c_1$ and $c_2$ would be $v_1v_2u$-distinguishing. Without loss of generality, assume $u \in V_1$. The proof is symmetric if $u \in V_2$.  

    By Claim~\ref{samepartdiffcolor}, there exists $v_3 \in V_3$ such that $\vec{v_3}(1) = \beta_1 \not \in \{1,2,\alpha_1\}$. Note that color $c_1$ is $v_1v_2v_3$-distinguishing. By Claim~\ref{distinguishing}, because color $c_2$ is not also $v_1v_2v_3$-distinguishing, we have $\vec{v_3}(2) \in \{1,2\}$. Likewise, color $c_1$ is $uv_2v_3$-distinguishing, and so $\vec{v_3}(2) \in \{\alpha_2,2\}$. Hence, $\vec{v_3}(2) = 2$.

    Similarly, by Claim~\ref{samepartdiffcolor}, there exists $v_3' \in V_3$ such that $\vec{v_3'}(2) = \beta_2 \not \in \{1,2,\alpha_2\}$. Because color $c_2$ is $v_1v_2v_3'$-distinguishing and $uv_2v_3'$-distinguishing, by Claim~\ref{distinguishing}, we have $\vec{v_3'}(1) = 2$.

    By the same argument, there exists $w_2 \in V_2$ such that $\vec{w_2}(1) = \gamma_1 \not \in \{1,2,\alpha_1,\beta_1\}$. Color $c_1$ is $v_1w_2v_3$-distinguishing so $\vec w_2(2) \in\{1,2\}$ and $v_1w_2v_3'$-distinguishing so $\vec w_2(2) \in \{1,\beta_2\}$. Thus, $\vec{w_2}(2) = 1$. This contradicts the fact that color $c_1$ is also $uw_2v_3'$-distinguishing, in which $\vec{w_2}(2) \in \{\alpha_2,\beta_2\}$. This completes the proof of the case $r=3$.
    \end{proof}
\begin{figure}
\centering
    \begin{tabular}{l||cc||l}
    &\multicolumn{2}{c||}{Component}&\multirow{2}{*}{Part}\\
    &$c_1$&$c_2$&\\\hline\hline
         $\vec{v_1}$&$1$&$1$&$V_1$  \\
         $\vec{v_2}$&$2$&$2$&$V_2$\\
         $ \vec{u}$&$\alpha_1$&$\alpha_2$&$V_1$\\
         $\vec{v_3}$&$\beta_1$&$2$&$V_3$\\
         $\vec{v_3'}$&$2$&$\beta_2$&$V_3$\\
         $\vec{w_2}$&$\gamma_1$&$1$&$V_2$\\
         
    \end{tabular}\label{fig:r=3}
    \caption{Components for colors $c_1$ and $c_2$ used in the proof of case $r=3$.}
        \end{figure}

\subsection{The case $r \ge 4$}

We now assume $r \ge 4$. Throughout the proof, the reader may refer to Figure~\ref{fig:rbig}.

Let $b_{r-1}$ and $b_r$ be vertices satisfying $\delta(b_{r-1}, b_r) \geq r$ by Claim~\ref{distr}. Without loss of generality, suppose $b_{r-1} \in V_{r-1}$ and $ b_r \in V_r$ with $\vec b_r = (1, 1, 1\ldots, 1)$ and $\vec b_{r-1} = (2,2, \ldots, 2, 1, 1, \ldots, 1)$, where there are $r-1 + q$ $2$s for some $q \geq 1$. Moreover, by Claim~\ref{smalldist}, we have $\delta(b_{r-1}, b_r) \leq t+1 < r+t$, and so $\vec b_{r-1}$ contains at least one 1. 

By Claim~\ref{samepartdiffcolor}, there exists vertices $b_1 \in V_1, b_2 \in V_2, \ldots, b_{r-2} \in V_{r-2}$ such that $\vec b_i(i) \notin \{1,2\}$ for all $i \in [r-2]$. There also exists, by Claim~\ref{t+1diff}, vertices $d$ and $d'$ (possibly $d=d'$) satisfying
\[\vec d(j) \neq \vec b_{1}(j) \text{ for } r\leq j \leq r-1 + q; \qquad \vec d(j) \neq 1 \text{ for } r-1+q < j \leq r+t,\]
\[\vec d'(j) \neq \vec b_{2}(j) \text{ for } r\leq j \leq r-1 + q; \qquad \vec d'(j) \neq 1 \text{ for } r-1+q < j \leq r+t.\]

The vertex $d$ belongs to the part $V_i$ for some $i$ and $d'$ to $V_{i'}$ for some $i'$. We consider the edges $e = \{b_1, \ldots, b_r\} - \{b_i\} \cup \{d\}$ and $e' = \{b_1, \ldots, b_r\} - \{b_{i'}\} \cup \{d'\}$ which receive colors say $c_{j^*}$ and $c_{k^*}$ respectively. Therefore the vectors of the $r$ vertices in $e$ must all agree in the $j^*$th coordinate and similar for $e'$.

\begin{claim} $d, d' \in V_{r-1} \cup V_r$. \end{claim}
\begin{proof}
If $d \in V_i$ where $i \in [r-2]$, then $e$ contains $b_{r-1}, b_r,$ and $d$. The vector $\vec d$ differs from $\vec b_r$ after the $(r-1 + q)$th coordinate, and $\vec b_{r-1}$ and $\vec b_r$ differ in coordinates $1$ through $r-1+q$, a contradiction. Thus coordinate $j^*$ cannot exist. Similar holds for $d'$.
\end{proof}

Thus $i, i' \in \{r, r-1\}$. Let $b_q$ be the unique vertex in $\{b_{r-1}, b_r\}-\{b_i\}$. We consider the $r$ vectors corresponding to vertices in edge $e$. Note that $b_q \in e$.

Vectors $\vec b_{1}$ and $\vec d$ differ in coordinates $r$ through $r-1 + q$, and $\vec d$ and $\vec b_{q}$ differ after coordinate $r-1 + q$. Moreover, for any $j \in [r-2]$, $b_j$ differs from $\vec b_q$ in coordinate $j$. Thus the only coordinate in which the vectors may all agree is $j^*= r-1$. We have $\vec b_1(r-1) = \ldots =\vec b_{r-2}(r-1) = \vec b_q(r-1) = \vec d(r-1)$. By a symmetric argument, $\vec b_1(r-1) = \ldots =\vec b_{r-2}(r-1) = \vec d'(r-1)$. This implies $\vec d(r-1) = \vec d'(r-1) =: \alpha$ where $\alpha = \vec b_q(r-1) = 1$ if $i = r-1$ and $\vec b_q(r-1) = 2$ if $i = r$, and $i = i'$.

Recall that we chose the vertex $b_j$ with $j \in [r-2]$ such that $\vec b_j(j) \notin \{1,2\}$. We will show that besides coordinate $j$, for $j \in [r-2]$, the first $r-1$ elements of $\vec b_j$ must each equal either 1 or 2. 

\begin{claim}\label{all2s}
For every $j \in [r-2]$ and $k \in [r-1]$ with $k \neq j$, we have $\vec b_j(k) =\alpha$.  Moreover, $\vec d(2)=\vec d(3) = \ldots  = \vec d(r-1) = \alpha$.
\end{claim}

\begin{proof}
    We have shown that $\vec b_j(r-1) =\alpha$ for all $j \in [r-2]$, and $\vec d(r-1) = \alpha$. First fix $2 \leq k \leq r-2$. By Claim~\ref{samepartdiffcolor}, we can find a vertex $b_{k}' \in V_{k}$ with $\vec b_{k}'(r-1) \notin \{1,2\}$. 

If $d \in V_r$, then the edge $f= \{b_1, b_2, \ldots, b_{r-1}, d\} - \{b_{k}\} \cup \{b_{k}'\}$ has some color $c_{j'}$, and so the $j'$th coordinate of the vectors corresponding to these $r$ vertices must be the same. By construction $j' \neq r-1$ since $\vec b'_k(r-1) \neq \vec b_{r-1}(r-1)$. We cannot have $j' > r-1+q$ as $\vec d$ and $\vec b_{r-1}$ differ in these coordinates, and we cannot have $r-1 < j' \leq r-1+q$ as $\vec d$ and $\vec b_{1}$ differ in these coordinates.
For any $1\leq j \leq r-2$, $j \neq k$,  the vector $\vec b_{j}$ differs from $\vec b_{r-1}$ in coordinate $j$. The only coordinate remaining is $j' = k$. In particular $\vec b_{r-1}(k) = 2 = \alpha$, so $\vec{d}(k) = \vec b_j(k) =\vec b_{r-1}(k) =\alpha$ for all $j \ne k$.  
Similarly we can show that if $d \in V_{r-1}$, then $\vec{d}(k)=\vec b_j(k) = \vec b_r(k) = 1 = \alpha$ for all $j \neq k$.

Finally, fix $b_1' \in V_1$ such that $\vec b_1'(r-1) \notin \{1,2\}$. If $ d \in V_r$, let $f' = \{b_1', b_2, \ldots, b_{r-1}, d'\}$, and let $k'$ be a coordinate in which the vectors of these $r$ vertices coincide. By a similar argument, $k' = 1$, and thus for all $j \neq 1, \vec b_j(1) = \vec b_{r-1}(1) = 2 = \alpha$. Moreover, if $d \in V_{r-1}$, then $\vec b_j(1) = 1 = \alpha$.
\end{proof}
 From now, we assume $d \in V_r$. The case where $d \in V_{r-1}$ is symmetric, so we omit it.

\begin{proof}[Proof of Theorem~\ref{main} for $r \ge 4$]
    
    Let $b_{r-1}' \in V_{r-1}$ be a vertex with $\vec b_{r-1}'(r-1) \notin \{1,2\}$. We consider the edge $\{b_1, \ldots, b_{r-2}, b_{r-1}', d\}$ whose $r$ vertices must agree in some coordinate $j'$. By the construction of $b_1, \ldots, b_{r-2}, b_{r-1}'$ and Claim~\ref{all2s}, we have $j'\notin [r-1]$. Moreover, $d$ and $b_{1}$ differ in coordinates $r$ through $r-1+q$. Thus $j' \geq r-1+q + 1$. Without loss of generality, let $j' = r+q$, and say \[\vec d(r+q) = \vec b_1(r+q) = \ldots = \vec b_{r-2}(r+q)=\vec b_{r-1}'(r+q)  =: \beta \neq 1,\] (since $\vec d(r+q) \neq 1$ by construction).

    Now let $d''$ be a vertex such that $\vec d''(r-1) \neq 2$, $\vec d''(i) \neq \vec b_{1}(i)$ for $r\leq i \leq r-1+q$, and $\vec d''(i) \neq 1$ for $r+q+1\leq i \leq r+t$. This involves $1+ q + (t+1-q - 1) = t+1$ coordinates, and hence such a vertex exists by Claim~\ref{t+1diff}.
    
    The vertex $ d''$ belongs to some set $V_i$. If $i =r$, then the vertices of edge $\{b_1, \ldots, b_{r-1}, d''\}$ must agree in some coordinate $j'$.  By construction, $j ' \notin [r-1+q]$. Moreover, $j' \notin \{r+q+1, \ldots, r+t\}$ since $\vec d''$ disagrees with $b_{r-1}$ in these coordinates. Hence $j' = r+q$. But $\vec b_{r-2}(r+q) = \beta \neq 1 = \vec b_{r-1}(r+q)$, a contradiction. One can prove $i \neq r-1$ similarly.

Thus we may assume $i \in [r-2]$. We consider the vectors of the vertices in the edge $\{b_1, \ldots, b_r\} - \{b_i\} \cup \{d''\}$ which must agree in some coordinate $j''$. Since $\vec b_{r-1}$ and $\vec b_r$ differ in their first $r-1+q$ coordinates, $j'' \geq r+q$. Moreover, $\vec b_r$ and $\vec d''$ differ in all coordinates from $r+q+1$ to $r+t$. Thus $j'' = r+q$. Since $r \neq 3$, then there exists some $k \in \{1, \ldots, r-2\} - \{i\}$. We obtain
\[1 = \vec b_{r-1}(r+q) = \vec b_r(r+q) = \vec d''(r+q) = \vec b_k(r+q) = \beta.\]
This contradicts that $\beta \neq 1$.
\begin{figure}[ht]
\centering
\begin{tabular}{l||ccccc|c|ccc|c|ccc||l}
     &\multicolumn{13}{c||}{Component}& \multirow{2}{*}{Part}\\
     &$c_1$&$c_2$&$c_3$&$\cdots$&$c_{r-2}$&$c_{r-1}$&$c_r$&$\cdots$&$c_{r-1+q}$&$c_{r+q}$&$c_{r+q+1}$&$\cdots$&$c_{r+t}$&\\
     \hline\hline
     $\vec b_1$&$*$&$2$&$2$&$\cdots$&2&$2$&&&&$\beta$&&&&$V_1$\\
     $\vec b_2 $&$2$&$*$&$2$&$\cdots$&2&$2$&&&&$\beta$&&&&$V_2$\\
     $\vec b_3 $&$2$&$2$&$*$&$\cdots$&2&$2$&&&&$\beta$&&&&$V_3$\\
     $\vdots$&$\vdots$&$\vdots$&$\vdots$&$\ddots$&$\vdots$&$\vdots$&&&&$\vdots$&&&&$\vdots$\\
    
     $\vec b_{r-2}$&$2$&$2$&$2$&$\cdots$&$*$&$2$&&&&$\beta$&&&&$V_{r-2}$\\
     $\vec b_{r-1}$&$2$&$2$&$2$&$\cdots$&$2$&$2$&$2$&$\cdots$&$2$&$1$&$1$&$\cdots$&$1$&$V_{r-1}$\\
     
     $\vec b_{r}$&$1$&$1$&$1$&$\cdots$&$1$&$1$&$1$&$\cdots$&$1$&$1$&$1$&$\cdots$&$1$&$V_{r}$\\
     \hline
     $\vec{d}$&$ $&$2$&$2$&$\cdots$&2&$2$&$\ne\vec b_{1}$&$\cdots$&$\ne\vec b_{1}$&$\beta$&$\neq 1$&$\cdots$&$\ne 1$&$V_{r}$\\
     
     $\vec{d'}$&$2$& & && & $2$ &$\ne\vec b_{2}$&$\cdots$&$\ne\vec b_{2}$& $\ne 1$&$\neq 1$&$\cdots$&$\ne 1$&$V_{r}$\\
     
     $\vec b_{r-1}'$&&&&&&$*$&&&&$\beta$&&&&$V_{r-1}$\\
     $\vec d''$&&&&&&$\ne 2$&$\ne \vec b_{1}$&$\cdots$&$\ne \vec b_{1}$&&$\neq 1$&$\cdots$&$\ne 1$&$\rightarrow\leftarrow$
\end{tabular}
\label{fig:rbig}
\caption{Vectors used in the proof of case $r \geq 4$.}
\end{figure}
\end{proof}

\section{Proof of Theorem~\ref{cov2k}}

Chen et al.~\cite{CFGLT} proved the following for a special class of $k$-edge-colorings.

\begin{obs}[\cite{CFGLT}]\label{cfglt1}Let $G$ be a biclique, and let $k\geq 2$ be an integer. For any $k$-edge-coloring of $G$, either $V(G)$ can be covered by at most $2k-2$ monochromatic components or each color class is a disjoint union of bicliques.
\end{obs}

\begin{thm}[Corollary 1 in~\cite{CFGLT}]\label{cfglt2}For $k \in \{2,3\}$ in any $k$-edge-coloring of a biclique $G$ in which every color class is a disjoint union of bicliques, there exists a covering of $V(G)$ using at most $k$ monochromatic components. Moreover, these components may be all of the same color.   
\end{thm}

{\bf Notation}. Suppose $G$ is a biclique with parts $X$ and $Y$. For $x \in V(G)$ and a color $i$, let $N_i(x)$ denote the set of color-$i$ neighbors of $x$. For $S \subseteq V(G)$, define $N_i(S) = \bigcup_{x \in S} N_i(x)$. Let $G_i(x)$ denote the monochromatic component of color $i$ in $G$ which contains $x$ (so $\{x\} \cup N_i(x) \subseteq G_i(x)$). For subsets $A \subseteq X$ and $B \subseteq Y$, let $G[A,B]$ denote the induced subgraph of $G$ with vertex set $A \cup B$. Note that $G[A,B]$ is also a biclique.\medskip

\begin{proof}[Proof of Theorem~\ref{cov2k}.]
    For $k=2$, by Observation~\ref{cfglt1} and Theorem~\ref{cfglt2}, we have that $2k-2 = k$ components suffice. We prove the statement for $k=3$. Suppose $G$ is a $3$-edge-colored biclique with parts $X$ and $Y$. If every color class of $G$ is a disjoint union of bicliques, then $k$ components suffice by Theorem~\ref{cfglt2}. Otherwise, without loss of generality, there is some $u \in X$ and $v \in Y$ in the same component of color class $1$ such that $uv$ receives color $2$.

    Because the coloring of $G$ is spanning, $N_3(u)$ and $N_3(v)$ are nonempty. The components in $\{G_1(u), G_2(u), G_3(u)\}$ cover $Y$, and $\{G_1(v), G_2(v), G_3(v)\}$ cover $X$. By assumption $G_1(u) = G_1(v)$ and $G_2(u) = G_2(v)$. If any edge of $G[N_3(u),N_3(v)]$ were of color $3$, then $G_3(u) = G_3(v)$ and $V(G)$ would have a $3$-component cover, namely, $\{G_1(u) , G_2(u), G_3(u)\}$. Hence, $G[N_3(u),N_3(v)]$ is a $2$-edge-colored biclique and admits a covering with at most $2$ monochromatic components. If the covering has only one component, then together with $G_1(u)$ and $G_2(u)$, we obtain a covering with 3 components. So suppose the 2 components covering $G[N_3(u),N_3(v)]$ are $C$ and $D$. Define $C_X = C \cap X, C_Y = C \cap Y$, and similar for $D_X, D_Y$. 
    
    {\em Case 1}: $C$ and $D$ are of the same color and each of $C_X, C_Y, D_X$ and $D_Y$ has at least one vertex. We may assume $C$ and $D$ have color 1. Additionally, we have that $G[C_X, D_Y]$ and $G[C_Y, D_X]$ are bicliques in color $2$, as otherwise $C$ and $D$ would be in the same color-$1$ component. Note that $N_1(C_X) \cap G_1(u) = \varnothing$, as otherwise, $G_1(u)$ and $C$ would be in the same color-$1$ component and $\{G_1(u), G_2(u), D\}$ covers $V(G)$.
    
    Thus, $G[C_X, G_1(u) \cap Y]$ and similarly $G[D_X, G_1(u) \cap Y]$ are $2$-edge-colored bicliques in colors $2$ and $3$. Consider $N_2(C_X) \cap G_1(u)$ and $N_2(D_X) \cap G_1(u)$. Suppose these sets were not disjoint. Then, there would be a vertex in $C_X$ and a vertex in $D_X$ in the same color-$2$ component of $G$. However, since $G[C_X, D_Y]$ and $G[C_Y, D_X]$ are bicliques in color $2$, we would have that $C$ and $D$ are in the same color-$2$ component of $G$ and we can find a covering using 3 components.
    
    Hence, we have that $N_2(C _X) \cap G_1(u) \subseteq N_3(D_X) \cap G_1(u)$ since $N_3(D_X) \cup N_2(D_X)$ covers $G_1(u) \cap Y$. This implies that $G_1(u) \cap Y \subseteq N_3(C_X) \cup N_3(D_X)$. A similar argument applied towards $G_2(u) \cap Y$ yields $ G_2(u) \cap Y \subseteq N_3(C_X) \cup N_3(D_X)$. As $C_X, D_X \subseteq N_3(v)\subseteq G_3(v)$ every vertex of $(G_1(u) \cup G_2(u)) \cap Y$ is contained in the component $G_3(v)$. The remaining vertices in $Y$ belong to the component $G_3(u)$, so $Y$ can be covered by 2 components in color $3$. For every $x \in X$, because the coloring of $G$ is spanning, we have that $N_3(x)$ is nonempty, and so $x$ is in one of the two color-$3$ components. Thus, $V(G)$ admits a cover by at most two color-$3$ components.\medskip

    {\em Case 2}: $C$ and $D$ are of different colors. Note that this also encompasses the case where $C$ and $D$ are the same color, say 1, and some of $C_X, C_Y, D_X, D_Y$ are empty. Indeed, if say $C_X$ is empty, then  since $C$ is a connected subgraph of a bipartite graph, $C_Y$ is a single vertex, and we may instead take $C$ to be the color-$2$ component containing the vertex.

    Without loss of generality, suppose $C$ is in color $1$ and $D$ in color $2$. If $C \subseteq D$ or $D \subseteq C$, then $G$ can be covered by at most 3 monochromatic components. So without loss of generality, we may assume $C_X - D_X \neq \varnothing$. If $D_Y - C_Y$ is also nonempty, then the edges between $C_X - D_X$ and $D_Y - C_Y$ must receive some colors, but they cannot be $1$ nor $2$ due to the maximality of $C$ and $D$, a contradiction. It follows that $D_Y \subseteq C_Y$ and $D_X - C_X \neq \varnothing$. Moreover, for any $w \in D_X -C_X$, all edges from $w$ to $C_Y$ are color $2$ and therefore $D_Y = C_Y$.

    Again, no edge of $G[C_X, G_1(u)]$ may receive color $1$. This subgraph is therefore a $2$-edge-colored biclique in colors $2$ and $3$. We claim that 
    \begin{equation}\label{N3} G_1(u) \cap Y \subseteq N_3(C_X) \cup N_3(D_X).\end{equation}Indeed, if $G_1(u) \cap Y \subseteq N_3(C_X)$, then we're done. So suppose there exists $w \in G_1(u) \cap Y$ such that $w \notin N_3(C_X)$. Then all edges from $w$ to $C_X$ are in color $2$. If $w$ is incident to any color-2 edges to $D_X$, then since $D_Y = C_Y$, $D$ and $C$ belong to the same color-2 component and $G$ admits a $3$-covering. If all edges in $G[w, D_X]$ are in color $1$, then $D_X \subseteq G_1(u)$ and the $G$ admits the covering $\{G_1(u), G_2(u), C\}$. It follows that whenever $w \notin N_3(C_X)$, we obtain $w \in N_3(D_X)$, proving~\eqref{N3}.

    The same claim can be proven for $G_2(u)$, namely that $ G_2(u) \cap Y \subseteq N_3(C_X)\cup N_3(D_X)$. As in the previous case, $N_3(C_X), N_3(D_X) \subseteq G_3(v)$. Thus, every vertex in $Y$ is in $G_3(v)$ or $G_3(u)$. Again, for every $x \in X$, we have $N_3(x) \ne \varnothing$ due to the spanning coloring of $G$, implying that $G$ admits a covering by at most 2 components in color $3$. Hence, $\cov(2,3) = 3$.
\end{proof}

\bigskip {\bf Acknowledgement}. We thank Andr\'as Gy\'arf\'as for his helpful comments and discussions. 

\printbibliography

\end{document}